\documentclass[english,reqno]{amsart}

\usepackage{amsmath}
\usepackage{amssymb} 
\usepackage[new]{old-arrows}
\usepackage{enumitem} 
\usepackage{mathtools} 
\usepackage[table]{xcolor} 
\usepackage[all]{xy} 
\usepackage{tikz} 
\usepackage{tikz-cd}
\usepackage{indentfirst} 
\usepackage{babel} 
\usepackage{setspace} 
\usepackage{stmaryrd}

\usepackage[colorlinks,linkcolor=red,anchorcolor=green,citecolor=blue]{hyperref} 
\hypersetup{linktocpage = true} 

\usepackage{rotating} 

\usepackage{ytableau} 
\usepackage{longtable} 
\newcolumntype{M}[1]{>{\centering\arraybackslash}m{#1}} 

\usepackage{mathpazo}
\usepackage{mathrsfs}
\DeclareFontFamily{OMS}{rsfs}{\skewchar\font'60}
\DeclareFontShape{OMS}{rsfs}{m}{n}{<-5>rsfs5 <5-7>rsfs7 <7->rsfs10 }{}
\DeclareSymbolFont{rsfs}{OMS}{rsfs}{m}{n}
\DeclareSymbolFontAlphabet{\scr}{rsfs}
\DeclareSymbolFontAlphabet{\scr}{rsfs}

\usepackage[T1]{fontenc}

\newcommand\cF{{\mathcal F}}

\theoremstyle{plain}
\newtheorem{thm}{Theorem}[section]
\newtheorem{lemma}[thm]{Lemma}
\newtheorem{prop}[thm]{Proposition}

\newtheorem{defn}[thm]{Definition}
\newtheorem{claim}[thm]{Claim}

\theoremstyle{remark}

\newtheorem{remark}[thm]{Remark}

\def\dim{\operatorname{dim}}

\def\hor{\operatorname{\hor}}
\def\ver{\operatorname{\ver}}

\def\sing{\operatorname{\textsubscript{\rm sing}}}
\def\orb{\mathrm{orb}}
\def\ver{\operatorname{\textsubscript{\rm ver}}}
\def\hor{\operatorname{\textsubscript{\rm hor}}}

\def\orb{\mathrm{orb}}

\setlist[itemize]{leftmargin=*}
\setlist[enumerate]{leftmargin=*}

\numberwithin{equation}{section} 

\makeatletter

\ifnum\@ptsize=0  \fi \ifnum\@ptsize=2  \fi

\title{Title} 

\subjclass[2010]{}

\author{Wenhao Ou}

\address{Wenhao Ou, Institute of Mathematics, Academy of Mathematics and Systems Science, Chinese Academy of Sciences, Beijing, 100190, China}
\email{wenhaoou@amss.ac.cn}

\begin{document}

\begin{abstract}
Assume that $X$ is a compact complex analytic variety which has quotient singularities in codimension 2, and that $\mathcal{F}$ is a reflexive sheaf on $X$.  
Using orbifold modifications, we can define first and second homological Chern classes for $\mathcal{F}$. 
If in addition   $X$ has a K\"ahler form $\omega$  and $\mathcal{F}$ is $\omega$-stable, then we deduce Bogomolov-Gieseker inequality on the orbifold Chern classes of $\mathcal{F}$.  
\end{abstract}

\title{Orbifold Chern classes and Bogomolov-Gieseker inequalities}

\maketitle

\tableofcontents


\section{Introduction}

The theory of  holomorphic vector bundles is a central object in complex algebraic geometry and complex analytic geometry. 
The notion of stable vector bundles on complete curves was introduced by Mumford  in  \cite{Mumford1963}. 
Such notion of stability was then extended to torsion-free sheaves on any projective manifolds (see \cite{Takemoto1972}, \cite{Gieseker1977}), and is now known as the slope stability. 
An  important property of stable vector bundles is the following Bogomolov-Gieseker inequality, involving the Chern classes of the vector  bundle.
\begin{thm}
\label{thm:BG-inequality-intro}
Let $X$ be a projective manifold of dimension $n$,   let  $H$ be  an ample divisor, and let $\cF$ be a $H$-stable vector bundle of rank $r$ on $X$. 
Then 
\[   \Big(c_2(\cF)-\frac{r-1}{2r}c_1(\cF)^2 \Big)  \cdot  H^{n-2} \ge  0. \]
\end{thm}
When  $X$ is a surface, the inequality was proved in \cite{Bogomolov1978}. 
In higher dimensions, one may apply Mehta-Ramanathan   theorem in  \cite{MehtaRamanathan1981/82}   to  reduce to the case of surfaces, by taking hyperplane  sections. 
Later in \cite{Kawamata1992}, as a part of the proof for the three-dimensional abundance theorem, Kawamata extended the inequality to orbifold Chern  classes of  reflexive sheaves on projective surfaces with quotient singularities. 
The technique of taking hypersurface sections then  allows us to deduce the Bogomolov-Gieseker inequality for reflexive sheaves on projective varieties which have quotient singularities in codimension 2.

On the analytic side, let $(X,\omega)$ be a compact  K\"ahler manifold, and $(\cF, h)$ a Hermitian  holomorphic vector bundle on $X$.  
L\"ubke proved that  if $h$ satisfies the Einstein condition, then the following inequality holds (see \cite{Lub82}),
\[   \int_X \Big(c_2(\cF,h)-\frac{r-1}{2r}c_1(\cF,h)^2 \Big)  \wedge  \omega^{n-2} \ge 0. \] 
It is now well understood that if $\cF$ is slope stable, then it admits a Hermitian-Einstein metric. 
The case when $X$ is a complete curve was proved by Narasimhan-Seshadri  in \cite{NarasimhanSeshadri1965}, the case of projective surfaces was proved by Donaldson in \cite{Donaldson1985}, and the case of arbitrary compact K\"ahler manifolds was proved by Uhlenbeck-Yau in \cite{UhlenbeckYau1986}.   
Simpson extended the existence of  Hermitian-Einstein  metric to stable Higgs bundles, on compact and certain non compact K\"ahler manifolds, see \cite{Simpson1988}. 
Furthermore,  in \cite{BandoSiu1994},  Bando-Siu introduced the notion of admissible metrics and proved the existence of admissible  Hermitian-Einstein  metrics on stable reflexive sheaves. 

Comparing with the algebraic version, it is natural to expect a Bogomolov-Gieseker  type inequality,  for  stable coherent reflexive sheaves  on a  compact K\"ahler variety, 
which has at most quotient singularities  in codimension 2, see for example \cite{CHP23}.  
When the underlying space has quotient singularities only, for example when it is a surface, an orbifold version of Donaldson-Uhlenbeck-Yau theorem was   proved by Faulk in \cite{Faulk2022}.   
As a consequence, a Bogomolov-Gieseker type inequality holds in this case. 
However, for a general K\"ahler variety, we are not able to take hyperplane sections. 
So the algebraic method does not apply.  
When the underlying space is smooth in codimension 2, 
Bogomolov-Gieseker type inequalities have also been established, 
see for example \cite{CHP16}, \cite{Chen2022}, \cite{ChenWentworth2024} and \cite{Wu21}.  
In general, it was suggested in \cite{CGNPPW} that the existence of orbifold modifications would imply the inequality.  

Such   orbifold modifications were recently proved  in  \cite{KollarOu2025}. 
In this note, we will use them to define (homological) orbifold Chern classes $\hat{c}_\bullet(\mathcal{E})$ for reflexive sheaves $\mathcal{E}$,  and prove the following orbifold Bogomolov-Gieseker  inequality.

\begin{thm}
\label{thm:BG} 
Let $(X,\omega)$ be a compact K\"ahler  variety of dimension $n \ge 2$, 
which has quotient singularities in codimension 2. 
Let $\mathcal{F}$ be a  $\omega$-stable reflexive coherent sheaf of rank $r\ge 2$ on $X$. 
Then we have the following inequality on orbifold Chern classes of $\mathcal{F}$,
\[
\Big( 2r \hat{c}_2(\mathcal{F}) - (r-1) \hat{c}_1(\mathcal{F})^2    \Big) \cdot [\omega]^{n-2} \ge 0.
\]
\end{thm}

We  refer to \cite{GuenanciaPaun2024} for a pure analytic approach to Bogomolov-Gieseker inequalities when $X$ is a klt threefold.   
The inequality of Theorem \ref{thm:BG} plays an important role in the abundance theorem for compact K\"ahler threefolds (see \cite{CHP23}, \cite{DasOu2023} and \cite{GuenanciaPaun2024}).  
We also remark that it is stronger than the one obtained by taking a desingularization of $X$. 
For example, when $X$ is a projective surface with quotient singularities, 
such a comparison was shown in \cite[Proposition 2.4.(6) and Theorem 5.1]{Langer2000}.

\vspace{2mm}

\noindent\textbf{Acknowledgment.} 
The author is grateful to Omprokash Das and J\'anos Koll\'ar for helpful conversations.
The author is supported by the National Key R\&D Program of China (No. 2021YFA1002300).

\section{Homological Orbifold Chern classes}

A complex orbifold $\mathfrak{X}$ with quotient space $X$ is defined by the following data. 
There is an open covering $\{U_i\}$ of $X$, there are complex manifolds $V_i$, there are  finite groups $G_i$ acting holomorphically on $V_i$, such that $V_i/G_i \cong U_i$. 
We require further that the $(V_i,G_i)$ are compatible along the overlaps.  
The orbifold structure $\mathfrak{X}$ is called standard if the actions of $G_i$ are free in codimension 1. 
For more details, we refer to, for example,   \cite[Section 3.1]{DasOu2023}.  
Throughout this paper, we always assume that the actions of $G_i$ are faithful.

Assume that  $X$ is a compact complex analytic variety of dimension $n$. 
Then there is a natural isomorphism  $H_{2n}(X,\mathbb{R}) \cong H_{2n}^{BM}(X,\mathbb{R})$, where $H_{2n}$ is the singular homology and $H_{2n}^{BM}$ is the Borel-Moore homology. 
We note that  $H_{2n}^{BM}(X,\mathbb{R}) \cong \mathbb{R}$ and 
it  is generated by a canonical element $[X]$, called the fundamental class of $X$.  
For more details on Borel-Moore homology, we refer to \cite[Chapter 19]{Fulton} and the references therein. 
As a consequence, for  cohomology classes $\sigma_1,...,\sigma_k \in H^{\bullet}(X, \mathbb{R})$, such that the sum of their degrees is equal to $2n$, we can define the following intersection number,
\[
\sigma_1  \cdots\sigma_k   :=  (\sigma_1 \smallsmile \cdots  \smallsmile  \sigma_k) \smallfrown [X] \in \mathbb{R}.
\]
When $X$ has only quotient singularities, this intersection form can be identified with the orbifold Poincar\'e duality, see \cite[Theorem 3]{Satake1956}.

Assume that  $X$ is a compact complex analytic variety of dimension $n$,  
which has quotient singularities in codimension 2.   
By an orbifold modification of $X$, we refer to a projective bimeromorphic morphism $f\colon Y\to X$ such that the indeterminacy locus of $f^{-1}$ has codimension at least 3 in $X$. 
Such an orbifold  modification always exists thanks to \cite[Theorem 1]{KollarOu2025}.  
There is a Zariski open subset $X^\circ $ of $X$, whose complement has codimension at least 3, such that there is a standard complex  orbifold $\mathfrak{X} = (X_j,G'_j)$ whose quotient space is $X^\circ$. 
Being standard means that the action of $G'_j$ is free in codimension 1.    
Let $\mathcal{F}$ be a coherent reflexive sheaf on $X$.  
Since reflexive sheaves on complex manifolds are locally free in codimension 2,  up to shrinking $X^\circ$, we can assume that $\mathcal{F}$ induces an orbifold vector bundle $\mathcal{F}_{\orb}$ on $\mathfrak{X}$, see \cite[Remark 3.4.(5)]{DasOu2023}.

Let $f\colon Y \to X$ be an orbifold modification, and let $ \mathfrak{Y}=(Y_i,G_i) $ be   the standard orbifold structure on $Y$, 
and let $\mathcal{E} = (f^*\mathcal{F})/(\mathrm{torsion})$.  
Then $\mathcal{E}$ induces a  reflexive  coherent orbifold sheaf $\mathcal{E}_{\orb}$ on $\mathfrak{Y}$,  so that on each orbifold chart $\pi_i\colon Y_i\to Y$, we have  $\mathcal{E}_i=(\pi_i^*(\mathcal{E}))^{**}$.   
By shrinking $X^\circ$, we may assume that $f$ is an isomorphism over it.

By \cite[Theorem 3.10]{DasOu2023}, there is a functorial projective bimeromorphic morphism $p_i\colon Z_i\to Y_i$, 
such that $Z_i$ is smooth, that $\mathcal{H}_i:= p_i^*\mathcal{E}_i/(\mathrm{torsion})$ is locally free, and that the indeterminacy locus of $p_i^{-1}$ has codimension at least 3. 
The functoriality implies that there is a canonical action of $G_i$ on $Z_i$ such that $p_i$ is $G_i$-equivariant. 
Furthermore, the  $(Z_i,G_i)$'s induce  a complex orbifold $\mathfrak{Z}$ with quotient space $Z$. 
The collection $(\mathcal{H}_i)$ defines an orbifold vector bundle $\mathcal{H}_{\orb}$  on $\mathfrak{Z}$.   
Therefore,  we have well-defined orbifold Chern classes $\hat{c}_1(\mathcal{H}_{\orb})$ and $\hat{c}_2(\mathcal{H}_{\orb})$ in $H^{\bullet}(Z,\mathbb{R})$, see  \cite[Section 2]{Bla96} or \cite[Definition 3.3]{DasOu2023}.
Let $q\colon Z\to X$ be the natural morphism.  
Then the indeterminacy locus  of $q^{-1}$ has codimension at least 3. 
We can now define the (homological) orbifold Chern classes $\hat{c}_1(\mathcal{F})$, $\hat{c}_2(\mathcal{F})$ and $\hat{c}_1(\mathcal{F})^2$.  

\begin{defn}
\label{defn:Chern} 
The orbifold Chern class $\hat{c}_1(\mathcal{F})$, $\hat{c}_2(\mathcal{F})$ and $\hat{c}_1(\mathcal{F})^2$ are defined as   linear forms on $H^{\bullet}(X,\mathbb{R})$, so that  
for any class $\sigma \in H^{\bullet}(X,\mathbb{R})$,  we have 
\begin{eqnarray*}
    \hat{c}_k(\mathcal{F}) \cdot \sigma = \hat{c}_k(\mathcal{H}_{\orb}) \cdot q^*\sigma &  \mbox{ and }  &  
\hat{c}_1(\mathcal{F})^2 \cdot \sigma = \hat{c}_1(\mathcal{H}_{\orb})^2 \cdot q^*\sigma, 
\end{eqnarray*}
where $k=1,2$. 
\end{defn}

To see that these classes are well-defined, it is equivalent to show that the intersection numbers are independent of the choice of the modification $q\colon Z\to X$.  
We first observe the following fact. 

\begin{lemma}
\label{lemma:c2-property}
Let  $X$ be compact complex analytic variety of  dimension $n$ and let $\rho \colon \widetilde{X} \to X$ be a proper bimeromorphic morphism.  
Let $\alpha_1,\alpha_2\in H^{2k}(\widetilde{X} ,\mathbb{R})$.  
Assume that there is a closed analytic subset $V\subseteq X$ of codimension at least $k+1$, such that 
$\alpha_1|_{\widetilde{X}^\circ} =  \alpha_2|_{\widetilde{X}^\circ}$, 
where  ${\widetilde{X}^\circ} =  \widetilde{X} \setminus \rho^{-1}(V)$.  
Then  for any class $\sigma \in H^{2n-2k}(X, \mathbb{R})$, we have 
\[
\alpha_1 \cdot \rho^* \sigma = \alpha_2 \cdot \rho^*\sigma. 
\]
\end{lemma}

\begin{proof} 
Let $E = \rho^{-1}(V)$. 
Then there is some class $\delta \in H^{2k}(\widetilde{X},\widetilde{X}\setminus E, \mathbb{R})$ such that 
  $\alpha_1 - \alpha_2$
is equal to the image of $\delta$ in $H^{2k}(\widetilde{X},\mathbb{R})$. 
It follows that 
\begin{eqnarray*} 
    (\alpha_1 - \alpha_2 ) \cdot \rho^*\sigma  
         &=& \delta \cdot \rho^*\sigma  \\
         &=&  (\rho^*\sigma \smallsmile \delta)  \smallfrown [\widetilde{X}] \\ 
         &=& \rho^*\sigma \smallfrown (\delta \smallfrown [\widetilde{X}]).
\end{eqnarray*}
We note that $\delta \smallfrown [\widetilde{X}] \in H_{2n-2k}(E, \mathbb{R})$. 
Then $\rho_*(\delta \smallfrown [\widetilde{X}]) \in H_{2n-2k}(V, \mathbb{R})$. 
By assumption, we have $\dim V \le n-k-1$. 
Hence $\rho_*(\delta \smallfrown [\widetilde{X}]) = 0$. 
By the projection formula,   we have 
\[
 \rho^*\sigma \smallfrown (\delta \smallfrown [\widetilde{X}]) = \sigma \smallfrown \rho_*(\delta \smallfrown [\widetilde{X}]) =0.
\]
This completes the proof of the lemma. 
\end{proof}

We are in position to prove the following statement. 

\begin{prop}
\label{prop:Chern}
The intersection numbers in Definition \ref{defn:Chern} are independent of the choice of $Z$. 
\end{prop}

\begin{proof}
We assume the following setting. 
Let  $q_j\colon Z_j \to X_j$ be proper bimeromorphic morphisms with $j=1,2$. 
There is a closed analytic subset  $V \subseteq X$ of codimension at least 3,  such that  $X^\circ := X\setminus V$ has quotient singularities.  
The restriction of the  reflexive coherent sheaf $\mathcal{F}$ on $X^\circ$ induces an orbifold vector bundle $\mathcal{F}_{\orb}$ on the standard orbifold structure of $X^\circ$. 
Let $Z_j^\circ = q_j^{-1}(X^\circ)$. Then $q_j$ is an isomorphism on $Z_j^\circ$. 
There is an orbifold $\mathfrak{Z}_j$ whose quotient space is equal to $Z_j$, such that $\mathfrak{Z}_j$ is standard over $Z_j^\circ$. 
There is an orbifold vector bundle $\mathcal{H}_{\orb, j}$ on $\mathfrak{Z}_j$, whose restriction over $Z_j^\circ$
is isomorphic to $\mathcal{F}_{\orb}$ over  $X^\circ$. 
Let $\alpha_j \in \{ \hat{c}_1(\mathcal{H}_{\orb,j}), \hat{c}_2(\mathcal{H}_{\orb,j}), \hat{c}_1(\mathcal{H}_{\orb,j})^2 \}$, such that $\alpha_1$ and $\alpha_2$ are the same type of characteristic class. 

To prove the proposition, it  is enough to show that, for any $\sigma \in H^{\bullet}(X,\mathbb{R})$, the following equality holds  
\[
\alpha_1 \cdot q_1^*\sigma  = \alpha_2 \cdot q_2^*\sigma. 
\]
We note that there is a natural  bimeromorphic  map  $\varphi\colon Z_1\dashrightarrow Z_2$ over $X$. 
Let $\widetilde{X} \subseteq Z_1\times Z_2$ be the closure of the  graph of $\varphi$,  
and let  $p_1\colon \widetilde{X} \to Z_1$,  $p_2\colon \widetilde{X} \to Z_2$ and $\rho\colon \widetilde{X} \to X$
be the natural morphisms. 
Then we have 
\[
\alpha_j \cdot q_j^*\sigma  = p_j^*\alpha_j \cdot \rho^*\sigma 
\]
for $j=1,2$. 
We notice that $p_1$, $p_2$ and $\rho$ are   isomorphisms over $X^\circ$. 
Hence if $\widetilde{X}^\circ = \rho^{-1}(X^\circ)$, then we have 
\[
(p_1^*\alpha_1)|_{\widetilde{X}^\circ} = (p_2^*\alpha_2)|_{\widetilde{X}^\circ}. 
\]
By applying Lemma \ref{lemma:c2-property}, we obtain that 
$\alpha_1 \cdot q_1^*\sigma  =\alpha_2 \cdot q_2^*\sigma.$
This completes the proof of the proposition. 
\end{proof}

\begin{remark}
\label{rmk:Chern} 
We gather some   properties on this notion of Chern classes, which can be derived directly from the definition and Lemma \ref{lemma:c2-property}. 
\begin{enumerate}
    \item  If $X$ has quotient singularities and if $\mathcal{F}$ induces an orbifold vector bundle $\mathcal{F}_{\orb}$ on the standard orbifold structure over $X$, then the Chern classes of $\mathcal{F}$ in Definition \ref{defn:Chern} coincides with the linear forms on $H^{\bullet}(X,\mathbb{R})$ induced by the orbifold  Chern classes of $\mathcal{F}_{\orb}$. 

    \item Let $X^\circ \subseteq X$ be an open subset with quotient singularities, such that $\mathcal{F}|_{X^\circ}$ induces an orbifold vector bundle $\mathcal{F}_{\orb}$ on the standard orbifold structure on $X^\circ$, 
    and that $X\setminus X^\circ$ has codimension at least 3. 
    Let $\Delta\in \{\hat{c}_1, \hat{c}_2, \hat{c}_1^2\}$. 
    If there is some  class $\beta \in H^{\bullet}(X, \mathbb{R})$, 
    whose restriction on $X^\circ$ is equal to $\Delta(\mathcal{F}_{\orb})$,  
    then $\Delta(\mathcal{F})$ is the same as the linear form on $H^{\bullet}(X, \mathbb{R})$ 
    defined by the intersection with  $\beta$. 

    \item Let $M = H_1\cap \cdots \cap  H_{n-2}$ be the complete intersection surface of basepoint-free Cartier divisors in       general position.  
          Then  $M $   has quotient singularities, $\mathcal{F}|_M$ is reflexive and inducing an orbifold vector bundle on the standard orbifold structure of $M$, and  we have 
          \[
          \Delta(\mathcal{F}) \cdot c_1(H_1)  \cdots  c_1(H_{n-2}) = \Delta(\mathcal{F}|_M),
          \]
          where $\Delta$ is either $\hat{c}_2$ or $\hat{c}_1^2$.

    \item Let $h\colon X'\to X$ be a desingularization and let $\mathcal{F}' = h^*\mathcal{F}/(\mathrm{torsion})$.  Then for any cohomology class $\sigma\in H^\bullet(X,\mathbb{R})$, we can define 
    \[
    c_1(\mathcal{F}) \cdot \sigma := c_1(\mathcal{F'})\cdot  h^*\sigma.
    \] 
    This homology Chern class $c_1(\mathcal{F})$ is independent of the choice of $h$, and we  have 
    \[
    \hat{c}_1(\mathcal{F})\cdot \sigma = c_1(\mathcal{F'})\cdot  h^*\sigma.
    \] 
\end{enumerate} 
\end{remark}

With the first Chern class in Definition \ref{defn:Chern}, we can extend  the notion of slope stability as follows.  
Let $X$ be a compact complex analytic  variety of dimension $n$,  
which has quotient singularities in codimension 2.  
Let $\mathcal{F}$ be a coherent reflexive sheaf on $X$.  
For a cohomology class $\alpha \in H^2(X,\mathbb{R})$, we can define the slope
\[
\mu_{\alpha} (\mathcal{F}) = \frac{\hat{c}_1(\mathcal{F}) \cdot \alpha^{n-1}}{\mathrm{rank}\, F}.
\]
Then $\mathcal{F}$ is called $\alpha$-semistable if for any nonzero coherent  subsheaf $\mathcal{E}\subseteq \mathcal{F}$, we have $\mu_{\alpha} (\mathcal{E}) \le \mu_{\alpha} (\mathcal{F})$. 
It is called $\alpha$-stable if the inequality is strict whenever $\mathcal{E}$ has smaller rank.  
From the item (2) of  Remark  \ref{rmk:Chern}, we see that this definition coincides with the classic one, 
if $X$ satisfies that  every reflexive coherent sheaf of rank 1 on it 
has a positive reflexive power which is locally free.

\section{Orbifold Bogomolov-Gieseker  inequality}

In this section, we will prove Theorem \ref{thm:BG}.  
The notion of K\"ahler spaces was introduced in \cite{Grauert1962}.  
For   K\"ahler orbifolds and orbifold coherent sheaves, see for example  \cite[Section 2]{Faulk2022},  \cite[Section 2]{Wu23} or \cite[Section 3.1]{DasOu2023}. 






\begin{proof}[{Proof of Theorem \ref{thm:BG}}]
Let $f\colon Y\to X$ be an orbifold modification.  
Then $Y$ has quotient singularities. 
It follows that $Y$ is the quotient space of an orbifold  $ \mathfrak{Y} = (Y_i,G_i)_{i\in I}$, 
such that the action of $G_i$ is free in codimension 1. 
Since $f$ is projective, $Y$ is a compact K\"ahler variety.      
Let $\mathcal{E} = (f^*\mathcal{F})/(\mathrm{torsion})$.  
Then $\mathcal{E}$ induces a  torsion-free  coherent orbifold sheaf $\mathcal{E}_{\orb}$ on $\mathfrak{Y}$,  so that on each orbifold chart $\pi_i\colon Y_i\to Y$, we have  $\mathcal{E}_i=(\pi_i^*(\mathcal{E}))^{**}$.   
By shrinking $X^\circ$, we may assume that $f$ is an isomorphism over it.  
We note that $\mathcal{E}$ is stable with respect to $f^*\omega$.  

By \cite[Theorem 3.10]{DasOu2023}, there is a functorial projective bimeromorphic morphism $p_i\colon Z_i\to Y_i$, 
such that $Z_i$ is smooth, that $\mathcal{H}_i:= p_i^*\mathcal{E}_i/(\mathrm{torsion})$ is locally free, and that the indeterminacy locus of $p_i^{-1}$ has codimension at least 3. 
In particular, there is a canonical action of $G_i$ on $Z_i$ such that $p_i$ is $G_i$-equivariant, 
and that the $(Z_i,G_i)$'s induce a complex orbifold $\mathfrak{Z}$ with quotient space $Z$.  
The collection $(\mathcal{H}_i)$ defines an orbifold vector bundle $\mathcal{H}_{\orb}$  on $\mathfrak{Z}$.  
There is an induced projective bimeromorphic morphism $p\colon Z \to Y$. 
In particular, $Z$ is a K\"ahler variety.  
Then there is an orbifold K\"ahler form $\eta$ on $\mathfrak{Z}$, see \cite[Lemma 1]{Wu23}. 
Let $q\colon Z\to X$ be the natural morphism. 
We note that $q^*\omega$  can be viewed as  a semipositive orbifold $(1,1)$-form on $\mathfrak{Z}$. 
Moreover,  $\mathcal{H}_{\orb}$ is stable with respect to $q^*\omega$. 

\begin{claim}
\label{claim:stable}
For any $\varepsilon>0$ small enough,  
$\mathcal{H}_{\orb}$ is stable with respect to the following orbifold K\"ahler form on $\mathfrak{Z}$  
\[
\alpha_\varepsilon:= q^*\omega + \varepsilon \eta. 
\]   
\end{claim} 

Admitting this claim for the time being.  
By \cite[Theorem 1]{Faulk2022}, $\mathcal{H}_{\orb}$ admits an orbifold Hermitian-Einstein metric with respect to $\alpha_{\varepsilon}$. 
Then, as  in   \cite[Theorem 4.4.7]{Kobayashi2014},   we have  
\begin{equation*}\label{eqn:BG}
\Big(2r\hat{c}_2(\mathcal{H}_{\orb}) - (r-1)\hat{c}_1(\mathcal{H}_{\orb})\Big)\cdot [\alpha_\varepsilon]^{n-2}  
\ge 0.
\end{equation*}
The limit of the LHS above, when $\varepsilon$ tends to 0, is equal to 
\begin{eqnarray*}
&& \Big(2r\hat{c}_2(\mathcal{H}_{\orb}) - (r-1)\hat{c}_1(\mathcal{H}_{\orb})\Big)\cdot [ q ^*\omega ]^{n-2} \\
&=& \Big(2r\hat{c}_2(\mathcal{F}) - (r-1)\hat{c}_1(\mathcal{F})^2 \Big) \cdot [\omega]^{n-2}
\end{eqnarray*}   
This completes the proof of the theorem.
\end{proof}

It remains to prove Claim \ref{claim:stable}.

\begin{proof}[{Proof of Claim \ref{claim:stable}}]
By the same argument of \cite[Lemma 3.16]{DasOu2023}, it is enough to show that  
there is a constant $C>0$ such that 
\[
\mu_{q^*\omega}(\mathcal{H'}_{\orb}) \le   \mu_{q^*\omega}(\mathcal{H}_{\orb})  - C
\]
for any proper nonzero coherent orbifold subsheaf $\mathcal{H'}_{\orb} \subseteq \mathcal{H}_{\orb}$. 
It is enough to find a constant $C>0$ so that 
\[
 \mu_{\omega}(\mathcal{F'} ) \le  \mu_{\omega}(\mathcal{F})  - C
\]
for any  proper  coherent subsheaf $\mathcal{F'}$ of $\mathcal{F}.$ 
As pointed out in the item (4) of Remark \ref{rmk:Chern}, the orbifold first Chern class
$\hat{c}_1$ is indeed equal to the homological first Chern class defined by taking desingularizations.   
Hence we can apply \cite[Lemma 5.7]{Toma2021} to show that,
the following set of numbers 
\[\{  \mu_{\omega}(\mathcal{F'} )  \ | \ \mathcal{F'}\subseteq \mathcal{F} \mbox{ and } \mu_{\omega}(\mathcal{F'} ) \ge \mu_{\omega}(\mathcal{F}) -1\}\]
is finite. 
The existence of $C$ then follows from the stability of $\mathcal{F}$.  
This completes the proof of the claim.  
\end{proof}

\bibliographystyle{alpha}
\bibliography{reference}

\end{document}